\documentclass[12pt]{amsart}

\usepackage{
amsfonts,
latexsym,
amssymb,
amsbsy,
enumerate,
mathrsfs,
}

\newcommand{\labbel}{\label}

\newtheorem{theorem}{Theorem}[section]
\newtheorem{lemma}[theorem]{Lemma}

\newtheorem{proposition}[theorem]{Proposition} 
 
\newtheorem{corollary}[theorem]{Corollary}

\newtheorem*{theorem*}{Theorem}
\newtheorem*{corollary*}{Corollary}
\newtheorem*{proposition*}{Proposition}

\theoremstyle{definition}

\newtheorem*{definition*}{Definition}

\theoremstyle{remark}
\newtheorem{remark}[theorem]{Remark}

\newtheorem*{disclaimer*}{Disclaimer}

\DeclareMathOperator{\cf}{cf}

\newcommand{\m}{\mathfrak}

\begin{document}
 
\title{A note on local properties in products} 

\author{Paolo Lipparini} 
\address{Locale del Dipartimento di Matematica\\Viale della Ricerca Scientifica\\II Universit\`a di Tor Vergata di Roma\\I-00133 ROME ITALY}
\urladdr{http://www.mat.uniroma2.it/\textasciitilde lipparin}

\keywords{Local, basic, product, topological space, open map} 

\subjclass[2010]{Primary 54D20;  Secondary 54B10}

\thanks{Work performed under the auspices of G.N.S.A.G.A}

\begin{abstract}
We give conditions
under which a product of topological spaces 
satisfies some local property.
The conditions are necessary and sufficient
when  the corresponding 
global property is preserved under finite products.
Further examples include local sequential compactness, 
local Lindel\"ofness, the local Menger property.
\end{abstract} 
 
\maketitle

\section{Introduction} \labbel{intro} 

Conditions under which a product of topological spaces 
satisfies some local property have long been 
known in many particular instances.
Results with a general flavor appeared in 
 Preu\ss\ \cite[Section 5.3]{P}
and in
Hoffmann \cite[Theorem 2.2 and   Remark 2.4(b)]{H}.
Notice that the terminology used by the above authors
sometimes differs from the one we shall use.
The results from
\cite{P, H} have been improved, together with significant 
examples and applications, in 
Brandhorst \cite{Br} and Brandhorst and  Ern\'e \cite{BE}.
We refer to \cite{BE} for historical remarks
and examples; in particular, 
about how the definitions and  the results generalize classical cases in special situations.
Here we give a complete characterization
of those spaces which are local relative to 
some class closed under finite products.
We also deal with some classes which are not even closed
under finite product. 

Let $\mathcal T$ 
 be a class of topological spaces.
Members of $\mathcal T$ will be called $\mathcal T$-spaces.
For each class $\mathcal T$, three  local notions are defined; see \cite{BE,H,P}
and Hoshina \cite{Henc}. 
A topological space $X$ is a \emph{local $\mathcal T$-space}
(resp., a \emph{basic $\mathcal T$-space}) if,
for every point $x \in X$ and every neighborhood 
 $U$ of $x$,
 there is a neighborhood (resp., an open neighborhood) $V$ of $x$ such that 
$V \subseteq U$ and $V \in \mathcal T$. 
We sometimes say that $X$ is $\mathcal T$-local,
instead of saying that $X$ is a local $\mathcal T$-space,
and similarly for $\mathcal T$-basic.
Under the Axiom of Choice (AC) 
a space is $\mathcal T$-local 
if and only if  every point of $x$ has a neighborhood base consisting
of $\mathcal T$-subspaces.
Here a \emph{$\mathcal T$-subspace}
is a subspace which belongs to $\mathcal T$;
similarly, a \emph{$\mathcal T$-neighborhood}
of some point is a neighborhood of
that point which belongs to $\mathcal T$.   
Again using AC,
a space is $\mathcal T$-basic if and only if 
it  has  an open base consisting 
of $\mathcal T$-subspaces.
 However, we shall try to avoid the use of AC
as much as possible; see Remark \ref{ac}.
As a rougher notion, 
a \emph{local$_1$ $\mathcal T$-space}
is a space such that each point has 
at least one neighborhood which is a $\mathcal T$-space.
In many cases, especially assuming some separation axiom,
localness and local$_1$ness coincide,
and sometimes all the three above local notions coincide,
but sometimes not. See Lemma \ref{triv}(d), where another local property 
frequently equivalent to $\mathcal T$-localness shall
be mentioned.  

Characterizations of basic and local $\mathcal T$-spaces
appear in \cite{H,P}, in the case when $\mathcal T$ is
closed with respect both to products and to 
images of 
surjective continuous functions.
A characterization under weaker conditions
appears in \cite[Theorem 2.4]{BE}.
In the quoted theorem
 $\mathcal T$ 
 has to be closed under finite products,
and a further condition has to be satisfied.
We prove here a more general statement  which applies to
\emph{every} class $\mathcal T$ which is closed under finite products.
The proof is perhaps  simpler.
Then a characterization is given for certain $\mathcal T$ which are not even closed under finite products.
 Examples include 
local sequential compactness and local Lindel\"ofness.
Moreover, we show that
the assumption that $\mathcal T$ is closed 
with respect to 
images of 
surjective continuous functions
can be considerably weakened.
We reformulate many results in such a way that,
seemingly, the Axiom of Choice is not needed.
No separation axiom is used, either,
unless  explicitly stated otherwise.
All products under consideration are endowed with the Tychonoff 
topology, the coarsest topology making all the projections continue.
Most results would change dramatically,
when considering the box topology, or intermediate topologies.

The next lemma is trivial, we shall use it
(especially clause (f))
 in many examples, generally  without 
explicit mention. $T_3$ means \emph{regular and Hausdorff},
while we do not assume that \emph{regular} implies Hausdorff.

\begin{lemma} \labbel{triv}
Let $\mathcal T$ be a class of topological spaces.
  \begin{enumerate}[(a)]
 \item  
$\mathcal T$-basic implies $\mathcal T$-local
and $\mathcal T$-local implies $\mathcal T$-local$_1$.
\item
A $\mathcal T$-space is  $\mathcal T$-local$_1$.
Hence,
for every class $\mathcal T$,
 $\mathcal T$-local and  $\mathcal T$-local$_1$ are equivalent
if and only if every $\mathcal T$-space is  $\mathcal T$-local.
\item
If $\mathcal T$ is open-hereditary,
then all the three local properties coincide,
in particular any 
$\mathcal T$-space is $\mathcal T$-basic and  $\mathcal T$-local.
\item 
If $\mathcal T$ is closed-hereditary, then,
in a regular topological space, $\mathcal T$-localness
and $\mathcal T$-local$_1$ness are equivalent,
 and are also equivalent to the following.
  \begin{enumerate}[(1)]    
\item[{\rm (L)}] 
{\rm For every point $x$ and every neighborhood $U$ of $x$,
there is an open neighborhood $V$ of $x$ such that 
$ \overline{V} \subseteq U $ and $ \overline{V} \in \mathcal T$.}   
  \end{enumerate} 
\item
In particular, if $\mathcal T$ is closed-hereditary, then
a regular $\mathcal T$-space is $\mathcal T$-local.
\item 
In both (d) and (e) above, the assumption that $\mathcal T$ is
closed-hereditary  can be weakened 
to $\mathcal T$ being hereditary with respect to 
regular closed subsets.
 \end{enumerate}
 \end{lemma}

\section{A weak assumption} \labbel{prel} 

We shall use almost everywhere the following assumption (W).
  \begin{enumerate}
  \item[(W)]
$\mathcal T$ is a class of topological spaces which satisfies
the following properties.
  \begin{enumerate}
  \item[(W1)] $\mathcal T$ is closed under homeomorphic images.
\item[(W2)] Whenever $A$, $B$ are arbitrary topological spaces, $a \in A$,
$b \in B$ and   $ T \subseteq A \times B$ is a $\mathcal T$-neighborhood 
of $(a, b)$,
  then there is $S \subseteq A$  which is a $\mathcal T$-neighborhood  of  $a$. 
\item[(W3)]  Whenever $A$, $B$ are nonempty topological spaces, 
$B'$ is a nonempty open subset of $B$, 
 $ A \times B'  \subseteq T \subseteq A \times B$ and $T \in \mathcal T$,
  then $A \in \mathcal T$. 
   \end{enumerate} 
   \end{enumerate} 

Notice that, in particular, (W3)
implies the following weaker property.

  \begin{enumerate}[(WWW)] 
   \item [(W3$'$)] 
If $A \times B \in \mathcal T$,
for some nonempty $A$, $B$,
then $A \in \mathcal T$.
  \end{enumerate} 

In particular (provided that $\mathcal T$ contains at least
one  nonempty space),
 (W1) and (W3$'$) imply that every one-element space is a
$\mathcal T$-space.
If $\mathcal T$ is closed under images of continuous 
surjections,
then (W) is verified. Indeed, under that assumption, 
(W1) is trivial and, as far as (W2) and (W3)
are concerned,
it is enough to consider $ \pi_1(T)$,
where $\pi_1$ is the canonical projection onto the first factor.
In another direction, (W) is verified also in case $\mathcal T$ is
 hereditary and closed under homeomorphic images.
In this case it is enough
to consider $T \cap (A \times \{ b \}) $,
where, to get (W3), we pick any $b \in B'$.
If we are working in the context of $T_1$ spaces
(i.e., all spaces in (W2) and (W3) are assumed to be $T_1$)
then it is enough to assume that $\mathcal T$ is
closed-hereditary and closed under  homeomorphic images.
In particular, in the context of $T_1$ spaces,
the class $\mathcal T$ of all normal spaces satisfies (W).
Hence Property (W) seems to be definitely very\/ {\emph{W}}\/eak. 
A few more conditions implying (W) shall be discussed near the end of this note.

If $\prod _{i \in I} X_i $
is a product of topological spaces and 
$ J \subseteq I$, then $\prod _{i \in J} X_i $ 
is called a \emph{subproduct} of $\prod _{i \in I} X_i $.
If $I \setminus J$ is finite, then
(with a slight abuse of terminology) we shall call
$\prod _{i \in J} X_i $  a 
\emph{cofinite subproduct} of $\prod _{i \in I} X_i $.

\begin{lemma} \labbel{ln}
Suppose that $X$ is a nonempty
product of topological spaces and $\mathcal T$ is a class of topological spaces 
closed under homeomorphic images.
  \begin{enumerate}
\item[(a)]
If $\mathcal T$ satisfies 
{\rm (W2)}  and $X$ is $\mathcal T$-local,
then all factors and all subproducts are $\mathcal T$-local.
\item[(b)]
If $\mathcal T$ satisfies 
{\rm (W3)},  $X$ contains a set $T \in \mathcal T$
and $T$ contains some nonempty open set, then some cofinite subproduct 
belongs to $\mathcal T$. 
\item[(c)]
If $\mathcal T$ is closed under finite products,
$\mathcal T$  satisfies {\rm (W3$'$)},
$X \in \mathcal T$ and
 all factors of $X$ 
are $\mathcal T$-local, then $X$ is $\mathcal T$-local.
\item[(d)]
If $\mathcal T$ is closed under finite products,
then the product of two local $\mathcal T$-spaces
is still $\mathcal T$-local.
  \end{enumerate}
 \end{lemma}

 \begin{proof}
(a) Let $X=\prod _{i \in I} X_i $ be nonempty and $\mathcal T$-local
and let $ \emptyset \not=J \subseteq I$. 
Since $\prod _{i \in I} X_i $ is nonempty,
then also $Y=\prod _{i \in J} X_i $ is nonempty.
We have to show that
$Y $ is $\mathcal T$-local
(we allow $|J|= 1$, so this case takes into account factors).
Let $H= I \setminus J$. If $H= \emptyset $, there is nothing to prove,
hence we can suppose that   $H\not= \emptyset $.
Let $B= \prod _{i \in H} X_i $. Notice that $X$
is homeomorphic to $Y \times B$, hence  $Y \times B$
is $\mathcal T$-local, since $\mathcal T$ is closed
under homeomorphisms. 

Let $a  \in Y$ and suppose that
$A$  is a neighborhood of $a$ in $Y $.
 Since  $X=\prod _{i \in I} X_i $ is nonempty,
then
also $B= \prod _{i \in H} X_i $  is nonempty;
pick $b \in B$.
Now $A \times B$ is 
a neighborhood of $(a,b)$ 
in $Y \times B$, which is $\mathcal T$-local,
hence there is $T \subseteq A \times B$
which is a $\mathcal T$-neighborhood   
of $(a,b)$.
By (W2), there is  $ S \subseteq A$
which is a $\mathcal T$-neighborhood of $a$.
The above argument works
for every $a \in Y $  
and every neighborhood $A$ of $a$, hence
we get that $Y$  is $\mathcal T$-local.

(b) Let
$X=\prod _{i \in I} X_i $, hence $T$ contains a basic nonempty open 
set of the form
$\prod _{i \in I} Y_i $,
where  each $Y_i$ is open in $X_i$,
and $Y_i =X_i$, for every $i \in J = I \setminus F$,
with $F$ finite. If $F= \emptyset $ then $T=X$
and we are done, so suppose that  
 $F \not= \emptyset $.
Take $A= \prod _{i \in J} X_i  $
and 
$B= \prod _{i \in F} X_i  $. 
Since $\mathcal T$ is closed under homeomorphisms,
we lose no generality if we
identify $X$ with
$A \times B$. 
Taking $B' = \prod _{i \in F} Y_i $,
we have that 
 $ A \times B'  \subseteq T \subseteq A \times B$, hence we
 can apply (W3) to get that the cofinite subproduct 
$A$ belongs to $\mathcal T$.

(c) 
If $x = ( x_i) _{i \in I}  \in X$ and $U$ is a neighborhood of $x$,
then $U$ contains a basic 
open set of the form
$\prod _{i \in I} U_i $,
where 
$x_i \in U_i$ for every $i \in I$, and
$U_i=X_i$,
for all indices except perhaps for indices 
in a finite set $F$.
By (W3$'$) and closure under homeomorphisms,
$C=\prod _{i \in I \setminus F } X_i \in \mathcal T$.
Since each factor is $\mathcal T$-local,
then, for every $i \in F $,
$x_i$ has  a $\mathcal T$-neighborhood $V_i \subseteq U_i$.  
Let $V_i = X_i$ for  $i \not\in F $.
Then $\prod _{i \in I} V_i $ is a neighborhood 
of $x$ contained in $U$.
Since $\prod _{i \in I} V_i $ is homeomorphic to the finite product
$C \times \prod _{i \in F} V_i  $ and since
$\mathcal T$ is closed under finite products and homeomorphisms,
then  $\prod _{i \in I} V_i  \in \mathcal T$, hence 
 $\prod _{i \in I} V_i  $  
is a neighborhood of $x$ as requested. 

(d) is 
similar and easier.
 \end{proof}    

Notice that (a) in Lemma \ref{ln}
holds also in case we give to $\prod _{i \in I} X_i $  
the box topology, but this is not necessarily the case for (b) and (c).

\section{Properties closed under products} \labbel{clpr} 

\begin{theorem} \labbel{bebis}
Suppose that $X$ is a nonempty product
and $\mathcal T$ is a class 
of topological spaces closed under finite products 
and satisfying {\rm (W)}.
Then the following conditions are equivalent
(conditions marked with an asterisk are equivalent under the further 
assumption that every $\mathcal T$-space is $\mathcal T$-local).
 \begin{enumerate}
   \item 
$X$ is $\mathcal T$-local.
\item
 Each factor is $\mathcal T$-local and
some cofinite subproduct is a $\mathcal T$-space.
\item[(3)*]
Some cofinite subproduct is a $\mathcal T$-space
and each of the remaining factors are $\mathcal T$-local.
  \end{enumerate} 

If, in addition, $\mathcal T$ is closed under arbitrary products, then the preceding conditions are
also equivalent to the following ones.
  \begin{enumerate}
    \item[(4)]
Every countable subproduct is $\mathcal T$-local.
\item[(5)]
Each factor is $\mathcal T$-local and
all but a finite number of factors are $\mathcal T$-spaces.
\item[(6)*]
All but a finite number of factors are $\mathcal T$-spaces
and the remaining factors are $\mathcal T$-local.
  \end{enumerate}
\end{theorem}

\begin{proof}
(1) $\Rightarrow $  (2) follows from Lemma \ref{ln}(a)(b).

 (2) $\Rightarrow $  (1) By  Lemma \ref{ln}(c), the cofinite subproduct given by (2)
is $\mathcal T$-local, and then $X$ is homeomorphic to a finite product of 
$\mathcal T$-local spaces, hence $\mathcal T$-local, by Lemma \ref{ln}(d). 

(2) $\Rightarrow $  (3) is trivial.

If (3) holds and every $\mathcal T$-space is $\mathcal T$-local,
then the cofinite subproduct given by (3) is $\mathcal T$-local, hence every factor
is $\mathcal T$-local, by Lemma \ref{ln}(a). Thus (3) $\Rightarrow $  (2).

(2) $\Rightarrow $  (5) follows by (W3$'$) and (W1);
 (5) $\Rightarrow $  (2) is immediate from the additional assumption.
Hence, under the additional assumption,
(1), (2) and (5) are equivalent.

(1) $\Rightarrow $  (4) follows from Lemma \ref{ln}(a).

If (4) holds, then, again by \ref{ln}(a),  all factors are $\mathcal T$-local.
Suppose by contradiction that (4) holds and (5) fails, thus there are infinitely many 
factors which are not $\mathcal T$-spaces. Choose a countable subfamily.
By (4), the subproduct of the members of such a family is $\mathcal T$-local.
Applying the already proved implication (1) $\Rightarrow $  (5)
to this countable subproduct, we get that all but finitely many members
of the subfamily are $\mathcal T$-spaces, a contradiction. 

The equivalence of (5) and (6) is immediate from the 
assumption that every $\mathcal T$-space is $\mathcal T$-local.
\end{proof}
 
Notice that the equivalence of (1) and (2) above improves
\cite[Theorem 2.4]{BE}.
This is because the assumptions in \cite[Theorem 2.4]{BE}
imply that $\mathcal T$ is closed under finite products, and, under 
the same assumptions, the last conclusion in \cite[Theorem 2.4]{BE} 
is equivalent to the product
having a cofinite subproduct in $\mathcal T$.

The versatility of Theorem \ref{bebis} and the broad 
range of validity 
of Property (W) are shown
by the 
samples presented in  the 
next two corollaries. In some cases the results are well-known.
Further examples can be found
in \cite{BE}; in some cases the results here are slightly more general. 
Following \cite{BE}, if $\kappa$ is an infinite cardinal,
we denote by $\mathcal T_ \kappa $  the class of all spaces which can be obtained as the union
of $< \kappa $ many $\mathcal T$-spaces.  Notice that 
if $\mathcal T$ is closed under finite products, 
then $\mathcal T_ \kappa $ is closed under finite products, too.

\begin{corollary} \labbel{ex1}
A nonempty product of topological spaces 
is locally Hausdorff 
if and only if all but finitely many factors are 
Hausdorff 
and all the remaining factors are locally 
Hausdorff.
The same holds when ``Hausdorff''
is replaced by any one of the following:
$T_3$, regular, Tychonoff.

If we work in the context of regular  spaces,
the same applies to compact, 
sequentially pseudocompact,
bounded, $\lambda$-bounded,
$D$-compact,
$D$-feebly compact (for some given ultrafilter $D$).
Here and below we can also consider 
the conjunction of any set of the above properties, in particular,
simultaneous $D$-compactness, for 
$D$ belonging to a given set of ultrafilters.

Without assuming separation axioms, a
 nonempty product of topological spaces 
is locally $D$-compact
if and only if all factors are locally $D$-compact and
 all but finitely many factors are 
$D$-compact.
The same applies when ``$D$-compact''
is replaced by  any of the above mentioned properties,
as well as by connected, path-connected, $H$-closed.

Relative to any of the above properties
a nonempty product is local if and only if
every countable subproduct is local.

If $\kappa$ is an infinite cardinal, 
a nonempty product is locally $\kappa$-sequentially compact
if and only if all factors are locally $\kappa$-sequentially compact and 
some cofinite subproduct is $\kappa$-sequentially compact.
The same applies when ``$\kappa$-sequentially compact'' is
replaced by  $\mathcal T_ \kappa $ (if $\mathcal T$ is
closed under finite products and $\mathcal T_ \kappa $ satisfies
(W)), or ``of cardinality $<\kappa$''.
\end{corollary}

Notice that, for example,  a Hausdorff compact
space is locally compact, but this is not necessarily true
without assuming the Hausdorff property.
Hence, in case we assume no separation axiom, we get only the weaker statements
in the third paragraph of Corollary \ref{ex1}.
In most cases
the Hausdorff property is not enough and regularity is needed.
As an example, if some space is $D$-feebly compact,
then the closure of every open set 
is $D$-feebly compact, that is,
$D$-feeble compactness is 
hereditary with respect to regular closed sets.
Hence, by Lemma \ref{triv}(f),  a regular $D$-feebly compact space is locally  
$D$-feebly compact, but, again, this is not necessarily the case,
without assuming some separation axiom.
Notice that in the context
of Tychonoff spaces, 
$D$-feebly compact spaces are usually called
\emph{$D$-pseudocompact}.

For certain properties, some slightly more refined 
results can be obtained. Local sequential compactness
shall be dealt with  in the next section.

\begin{corollary} \labbel{metr} 
(a) A nonempty product is locally 
metrizable if and only if 
all but countably many factors are one-element,
all but finitely many factors are metrizable
and the remaining factors are locally metrizable.
In particular, a nonempty product is locally 
metrizable if and only if each subproduct
by $ \leq \omega_1$ factors is locally metrizable.

(b) A nonempty product is locally finite if and only if 
all but a finite number of factors are one-element spaces
and the remaining factors are locally finite.
A nonempty product is locally 
finite if and only if each countable subproduct
is locally finite.

The same applies when ``finite'' is replaced by either  ``countable'', 
or
``of cardinality $< \kappa $'', if $ \omega \leq \kappa \leq 2^ \omega $
(of course, this adds nothing, if the Continuum Hypothesis holds).  
 \end{corollary}

\section{Local sequential compactness} \labbel{examples} 

We first present another  corollary of Theorem \ref{bebis}.
 It deals with the general situation in which
a product belongs to $\mathcal T$ if and only if 
all subproducts by a small number of factors belong to $\mathcal T$.

\begin{corollary} \labbel{cor}
Suppose that $\mathcal T$ is a class 
of topological spaces closed under finite
products, $\mathcal T$ satisfies {\rm (W)} and  
 there is some
cardinal $\kappa > \omega $ such that 
a product belongs to $\mathcal T$ if and only if 
every subproduct by $<\kappa$ factors belongs to $\mathcal T$.

If $X=\prod _{i \in I} X_i $  is a nonempty product,
then the following conditions are equivalent.
 \begin{enumerate}
   \item[(I)] 
$X$ is $\mathcal T$-local.
\item[(II)]
Every subproduct by $<\kappa$ factors is  $\mathcal T$-local.
  \end{enumerate} 
 \end{corollary} 

\begin{proof}
(I) $\Rightarrow $  (II) follows from Lemma \ref{ln}(a).

We shall show that (II) implies Condition (2)
in Theorem \ref{bebis}. If (II) holds, then 
all factors are $\mathcal T$-local, again by  Lemma \ref{ln}(a).
Arguing as in the last part of the proof of Theorem \ref{bebis}
and since $\kappa$ is uncountable, we get that all but a finite number
of factors are $\mathcal T$-spaces. Let $J$ be the set 
of those factors which are in $\mathcal T$.
By assumption, any subproduct  $\prod _{i \in H} X_i $ of $X$ 
such that  $|H|<\kappa$ is $\mathcal T$-local, in particular, this happens if
$H \subseteq J$.
By Theorem \ref{bebis}(1) $\Rightarrow $  (2)
\emph{applied to  the product} $\prod _{i \in H} X_i $,
we get that $\prod _{i \in H'} X_i $ is a $\mathcal T$-space, 
for some $H'$ cofinite in $H$. If $H \subseteq J$,
then $X_i$ is a $\mathcal T$-space, for  $i \in H \setminus H'$,
hence,
since, by assumption, $\mathcal T$ is closed under finite products,
$\prod _{i \in H} X_i $ is a $\mathcal T$-space.
Since this happens for every $H \subseteq J$ such that 
$|H|<\kappa$, we get from the assumption on $\mathcal T$ that
$\prod _{i \in J} X_i $ belongs to $\mathcal T$.
Thus \ref{bebis}(2) holds.  
 \end{proof}  

\begin{corollary} \labbel{lsc}
 Let $X = \prod _{i \in I} X_i $  be a nonempty product.
Then the following conditions are equivalent.
 \begin{enumerate}
   \item 
$X$ is locally sequentially compact;
\item
each factor is locally sequentially compact and some cofinite subproduct
is sequentially compact;
\item
each factor is locally sequentially compact and 
there is a cofinite $J \subseteq I$ such that whenever 
$J' \subseteq J$ and $|J'| \leq \m s$, then    $\prod _{i \in J'} X_i $
is sequentially compact; 
\item
all subproducts by $\leq \m s$ factors are 
 locally sequentially compact;
\item 
($ \m h =  \m s$)
all factors are locally sequentially compact,
all but a finite number of factors are sequentially compact
and the set of factors with a nonconverging sequence
has cardinality $< \m s$.

\item 
($ \m h =  \m s$, for $T_1$ spaces)
all factors are locally sequentially compact,
all but a finite number of factors are sequentially compact,
and
the set of factors  with more than one point has cardinality $< \m s$.

\item 
($ \m h =  \m s$, for $T_3$  spaces)
the set of factors  with more than one point has cardinality $< \m s$,
all but a finite number of factors are sequentially compact, and the remaining factors are
locally sequentially compact.
  \end{enumerate} 
 \end{corollary}

 \begin{proof}
In \cite[Corollary 6.4]{L} we have proved that
a product is sequentially compact
if and only if all subproducts by $\leq \m s$ factors are
sequentially compact. 
See \cite{L} for the definition of $\m s$, $\m h$ and further references.

(1) $\Leftrightarrow $  (2)  is a particular case of 
the corresponding equivalence in Theorem \ref{bebis}.

(2) $\Leftrightarrow $  (3) follows from \cite[Corollary 6.4]{L}.

(1) $\Leftrightarrow $  (4)  follows from \cite[Corollary 6.4]{L}
and Corollary \ref{cor} with $\kappa= \m s^+$. 

In \cite[Corollary 6.6]{L} 
we have proved that if $ \m h =  \m s$, then 
a product is sequentially compact
if and only if 
all factors are sequentially compact and the set of factors
with a nonconverging
sequence has cardinality $<\m s$. This implies 
(2) $\Leftrightarrow $  (5).

(5) $\Leftrightarrow $  (6)
follows from the fact that a $T_1$ space in which every sequence converges is necessarily a one-point space.

(6) $\Leftrightarrow $  (7) follows from the fact that 
a $T_3$ sequentially compact space is locally sequentially compact.
 \end{proof}

\section{Some classes which are not closed under products} \labbel{clnpr}

In order to work with classes which
are not necessarily closed under products,
we shall consider the following property of some class $\mathcal T$.

  \begin{enumerate}
  \item [(S)] There are a class $\mathcal S$ of topological spaces 
and an infinite cardinal $\kappa$ such that  
a nonempty product 
$\prod _{i \in I} X_i $ belongs to $\mathcal T$ if and only if
$I$ can be written as a disjoint union $I=J \cup K$ in such a way that
$|J| < \kappa $, $\prod _{i \in J} X_i $ is a $\mathcal T$-space 
and $\prod _{i \in K} X_i $ is an $\mathcal S$-space.
We also require that $\mathcal S$ is closed
under homeomorphic images and  under taking cofinite subproducts.
  \end{enumerate} 

In the above condition we allow both $J= \emptyset $ and $K= \emptyset $.
This is consistent, since if $\mathcal T$ satisfies (W3$'$),
then any one-element space is a $\mathcal T$-space.
Moreover, ``$\mathcal S$ being closed under cofinite subproduct''
can be interpreted in a sense that 
 it implies that any one-element space belongs to $\mathcal S$.
In particular, (S) implies that every $\mathcal S$-space is a $\mathcal T$-space
and, more generally, that the product of a $\mathcal T$-space
with an $\mathcal S$-space is a $\mathcal T$-space.
Hence also the product of a $\mathcal T$-space with finitely many
$\mathcal S$-spaces is a $\mathcal T$-space. 
If not otherwise mentioned, we \emph{do not} require that $\mathcal S$
satisfy any special further property.

However, we should mention that if $\mathcal S$
satisfies the additional assumption
that a nonempty product belongs to $\mathcal S$
if and only if each factor belongs to $\mathcal S$  
 then 
a nonempty product belongs to $\mathcal T$ if and only if
every subproduct by $\leq \kappa$ factors belongs to $\mathcal T$.
Indeed, if the latter is the case, we cannot have $\kappa$-many factors
failing to be $\mathcal S$-spaces, hence the product
is a $\mathcal T$-space, by (S).

\begin{theorem} \labbel{st}
Suppose that  $\mathcal T$ is a class 
of topological spaces and $\mathcal T$ 
satisfies {\rm (W)} and {\rm (S)}, as given by $\mathcal S$ and $\kappa$.
If $X=\prod _{i \in I} X_i $  is a nonempty product, then the following conditions are equivalent. \begin{enumerate}
  \item 
$X$ is $\mathcal T$-local.
\item
Both the following conditions hold.
  \begin{enumerate}[(a)]
    \item    All subproducts of $X$ by 
$<\kappa$ factors 
are $\mathcal T$-local, and
\item
the index set  $I$ can be partitioned into two disjoint subsets
as $I=H \cup K$ in such a way that
$|H| < \kappa $
and $\prod _{i \in K} X_i $ is an $\mathcal S$-space.
\end{enumerate}
\item
The index set  $I$ can be partitioned into two disjoint subsets
as $I=H \cup K$ in such a way that
$|H| < \kappa $,
$\prod _{i \in K} X_i $ is an $\mathcal S$-space and
$\prod _{i \in H \cup F} X_i $ is  $\mathcal T$-local, for every 
finite $F \subseteq I$. 
  \end{enumerate} 

If  $\mathcal S$
satisfies the additional assumption
that a nonempty product belongs to $\mathcal S$
if and only if each factor belongs to $\mathcal S$,  
 then the preceding conditions (1)-(3) are equivalent to the following.

  \begin{enumerate}  
  \item [(4)]
All subproducts by $\leq \kappa$ factors are $\mathcal T$-local.
  \end{enumerate} 
 \end{theorem} 

\begin{proof} 
If (1) holds, then each subproduct  is a $\mathcal T$-space by Lemma \ref{ln}(a),
hence (2)(a) holds.
Moreover, by Lemma \ref{ln}(b), some cofinite subproduct
is a $\mathcal T$-space, hence (2)(b) follows from (S),
since if $F$ is finite and $|J| < \kappa $ then
$|J \cup F| < \kappa $, $\kappa$ being infinite.

(2) $\Rightarrow $  (3) is trivial.

Suppose that (3) holds,  $x = ( x_i) _{i \in I} \in X$ and
$U$ is a neighborhood of $x$. Thus $U$ 
contains a basic neighborhood 
of the form $\prod _{i \in I} U_i $,
where $U_i = X_i$, except for those $i$
in some finite set $F \subseteq I$.    
If $H$ and $K$ are given by (3), then,
by  the last requirement in  (S),
 $\prod _{i \in K \setminus F} X_i $ is an $\mathcal S$-space.
By (3), the subproduct 
$X'= \prod _{i \in H \cup F} X_i $ is $\mathcal T$-local. 
Consider the neighborhood   
$U'=\prod _{i \in H \cup F} U_i $
of $x'=( x_i) _{i \in H \cup F} $ in $X'$. 
Since $X'$ is $\mathcal T$-local, we get some $T \in \mathcal T$ 
such that $x' \in T \subseteq U'$.
By (S), $T \times  \prod _{i \in K \setminus F} X_i $ is a $\mathcal T$-space
and, modulo the natural homeomorphism, it is a neighborhood 
of $x$ contained in $U$.
Hence we have proved that $X$ is $\mathcal T$-local, that is (1) holds.

Thus (1)-(3) are equivalent.

(1) $\Rightarrow $  (4) follows again by Lemma \ref{ln}(a).

 We shall conclude the proof by showing that (4) implies (2),
under the additional assumption.
The implication (4) $\Rightarrow $  (2)(a) is trivial.
In order to show (2)(b), in view of the additional hypothesis,
it is enough to show that the set of all factors which are not $\mathcal S$-spaces
has cardinality $<\kappa$. Suppose by contradiction that 
$J \subseteq I$, $|J| = \kappa $ and $X_i \not \in \mathcal S $, for every $i \in J$.   
By (4), the subproduct $\prod _{i \in J} X_i $ is $\mathcal T$-local,
but then we get a contradiction by applying (1) $\Rightarrow $  (2)(b)
to that subproduct.
\end{proof}  

If $\prod _{i \in I} X_i $ is a product of topological spaces and $J \subseteq I$,
we shall say, again with some abuse of terminology, that a product
$\prod _{i \in H} X_i $  is a \emph{finite superproduct}
of $( X_i) _{i \in J} $ if $H=J \cup F$, for some finite $F \subseteq I$.   

\begin{corollary} \labbel{fin} 
 Suppose that $ n <\omega$ and $X$ is a nonempty product.
Then the following conditions are equivalent. 
\begin{enumerate}
   \item 
$X$ is locally finally $ \omega_n$-compact. 
\item
All but $ <\omega_n$ factors are compact,
and any finite superproduct of the set of noncompact factors
is locally finally $ \omega_n$-compact.
\item
Every subproduct by $\leq \omega _n$ factors 
is locally  finally $ \omega_n$-compact. 
\item 
(for $T_2$ spaces)
All but $ <\omega_n$ factors are compact,
and the product of the  noncompact factors
is locally finally $ \omega_n$-compact.
 \end{enumerate}

If $\lambda$ is a strong limit cardinal with $\cf\lambda \geq \omega _n$,
then all the above conditions hold when 
final $ \omega_n$-compactness is everywhere replaced
by $[ \omega _n, \lambda ]$-compactness
and compactness is replaced by initial $\lambda$-compactness
(but the separation assumption in (4) should be $T_3$).   
\end{corollary}

 \begin{proof}
Immediate from Theorem \ref{st}
and \cite[Theorems 4.1 and 4.3]{L}.  
 \end{proof} 

Notice that $ \omega_1$-final compactness is the same as Lindel\"ofness.
Since the product of countably many copies of $ \omega$
with the discrete topology is Lindel\"of and locally 
Lindel\"of, but the product of uncountably many copies
of $ \omega$ is not Lindel\"of (hence not locally Lindel\"of, either),
we get that ``$\leq \omega _1$'' in Condition (3) above
cannot be improved to ``$< \omega _1$''. 

However, we do not know whether
Corollary \ref{fin} can be improved, say, in the case of Lindel\"ofness,
to the following. A product is locally Lindel\"of if and only if
all but countably many factors are compact,
 all but finitely many factors are Lindel\"of 
and every finite subproduct is locally Lindel\"of.
We expect the above statement to be false, in general.

Again applying Theorem \ref{st},
in this case together with \cite[Corollary 5.3 and Propositions 5.1 and 5.2]{L}, we get the following.

\begin{corollary} \labbel{meng} 
If $X$ is a nonempty product,
then the following conditions are equivalent. 
\begin{enumerate}
   \item 
$X$ is locally Menger. 
\item
All but countably many factors are compact,
and any finite superproduct of the set of non Menger  factors
is locally Menger.
\item
Every subproduct by $\leq \omega _1$ factors 
is locally  Menger. 
\item 
(for $T_2$ spaces)
All but countably many factors are Menger,
and the product of the  non Menger factors
is locally Menger.
 \end{enumerate}
All the above conditions hold when 
Menger  is everywhere replaced
by either the Rothberger property,
or  the Rothberger property for countable covers,
and compactness by supercompactness.
\end{corollary}

\section{Further remarks} \labbel{fr} 

All the above arguments,
with the obvious modifications, can be applied
also to the ``basic'' and the  ``local$_1$'' case.

\begin{proposition} \labbel{same}
 Lemma \ref{ln}, Theorems \ref{bebis} and \ref{st} 
and Corollary \ref{cor}  hold
with ``local'' replaced everywhere by
 either ``basic'' or ``local$_1$'', except that in the ``basic'' case 
Condition {\rm (W2)} should be replaced everywhere by the following
Condition {\rm (W2$_O$)}, and {\rm (W)} should 
be modified accordingly, that is, we should consider
{\rm (W$_O$)}, the conjunction of {\rm (W1)}, {\rm (W2$_O$)} and {\rm (W3)}.
  \begin{enumerate}
\item[{\rm (W2$_O$)}] Whenever $A$, $B$ are  topological spaces, $a \in A$,
$b \in B$ and   $ T \subseteq A \times B$ is an \emph{open}  $\mathcal T$-neighborhood 
of $(a, b)$,
  then there is $S \subseteq A$  which is an \emph{open} $\mathcal T$-neighborhood  of  $a$. 
  \end{enumerate} 
 \end{proposition}  

Let us say that $\mathcal T$ satisfies
(C)
if $\mathcal T$ is closed under 
images of continuous surjection.
As we mentioned, (C) implies (W).
It is easy to see that
if $\mathcal T$ satisfies (C), then
the image of a local $\mathcal T$-space
under a continuous open map
is still a local $\mathcal T$-space.
In order to get the above conclusion, 
it is not enough to assume (W) in place of (C).
E.~g., the image of a Hausdorff 
space (hence locally Hausdorff)
is not necessarily locally Hausdorff.
The example is classical:  take two disjoint copies of the unit real interval
and pairwise identify the copies of $0$, as well as the 
copies of $1/n$, for each $n>0$. 
  
However, there are conditions weaker than (C) 
which still imply that images of local $\mathcal T$-spaces
under open continuous maps are $\mathcal T$-local.

  \begin{enumerate}
   \item[{\rm (C$^{-}$)}]
Whenever $X$ is a topological space, $T \subseteq X$ is a subspace,
$T \in \mathcal T$  and $\pi: X \to Y$ is a continuous open surjection,
then $\pi(T) \in \mathcal T$.  
   \item[{\rm (C$^{=}$)}]
Whenever $X$ is a topological space, $T \subseteq X$ 
contains some open set of $X$ 
and $\pi: X \to Y$ is a continuous open surjection,
then $\pi(T) \in \mathcal T$.  
   \item[{\rm (C$^{\equiv}$)}]
Whenever $X$ is a topological space, $x \in T \subseteq X$, 
$T  $ is a $ T$-neighborhood of $x$  in $X$ 
and $\pi: X \to Y$ is a continuous open surjection,
then $\pi(x)$ has some $\mathcal T$-neighborhood.  
   \end{enumerate}

Notice that (C) $\Rightarrow $  (C$^{-}$)
$\Rightarrow $    (C$^{=}$) $\Rightarrow $ 
(C$^{\equiv}$) $\Rightarrow $  (W2) and 
 (C$^{=}$) $\Rightarrow $ (W).

Consider also the following property (C$^{\equiv}_O$),
 which implies  (W2$_O$).

  \begin{enumerate}
   \item[{\rm (C$^{\equiv}_O$)}]
Whenever $X$ is a topological space, $x \in T \subseteq X$, 
$T \in \mathcal T$ is an open  neighborhood of $x$  in $X$ 
and $\pi: X \to Y$ is a continuous open surjection,
then $\pi(x)$ has some open neighborhood in $\mathcal T$.  
   \end{enumerate}

\begin{lemma} \labbel{lem}
If $\mathcal T$ is a class of topological spaces
satisfying {\rm (C$^{\equiv}$)},
then the  image of any local (resp., local$_1$) $\mathcal T$-space under a continuous
open surjection is a local (resp., local$_1$) $\mathcal T$-space. 

If $\mathcal T$ is a class of topological spaces
satisfying {\rm (C$^{\equiv}_O$)},
then the  image of any basic $\mathcal T$-space under a continuous
open surjection is a basic $\mathcal T$-space. 
 \end{lemma}

\begin{remark} \labbel{ambient}
We have usually worked in the class of arbitrary topological spaces,
however, essentially all the above definitions and results can be considered as
restricted to some special class, e.~g., $T_1$, Hausdorff or Tychonoff 
spaces. Seemingly, we can allow also spaces with a richer structure, e.~g., 
topological groups. We only need an ambient in which
it makes sense to talk of (arbitrary) products, and, if there is more structure
other than topology,
the topological Tychonoff product agrees with the product of the structure.
If we work in a specific ambient, say, of Hausdorff spaces, 
everything should be interpreted relative to that ambient;
for example, in that context, a class $\mathcal T$ 
is ``closed under images of 
surjective continuous functions'' 
if whenever $f:X \to Y$ 
is continuous and surjective, $X \in \mathcal T$
\emph{and $X$ and $Y$ are Hausdorff}, then
$Y \in \mathcal T$.  
For example, the class of Hausdorff compact spaces is 
closed under images of 
surjective continuous functions in the Hausdorff context, but \emph{not}
in the context of arbitrary topological spaces. 
 \end{remark}

\begin{remark} \labbel{improv}
It seems that, whenever we use the assumption
that $\mathcal T$ is closed under finite products,
we can do with the following weaker condition.
  \begin{enumerate}
    \item [(FP)] 
Whenever $A, B \in \mathcal T$
and $x \in A \times B$, then $x$ has a neighborhood 
in $\mathcal T$.  
 \end{enumerate}
This remark applies, e.~g.,  to 
Lemma \ref{ln}(c)(d), 
Theorem \ref{bebis}(1)-(3)
and Corollary \ref{cor}.
Notice that (FP) can be reformulated
as ``the product of two $\mathcal T$-spaces
is $\mathcal T$-local$_1$''.
Notice also that if $\mathcal T$ is such that 
every $\mathcal T$-space is $\mathcal T$-local,
then (FP) is equivalent to the assertion that the product
of two  $\mathcal T$-local spaces is $\mathcal T$-local.  
We know no application of the above remarks,
hence we have kept the statements in the simpler (but less general)
form.
 \end{remark}

\begin{remark} \labbel{ac}
Concerning our use of the Axiom of Choice (AC),
as the results are formulated,
it seems unnecessary in the statements and proofs of 
Lemmas \ref{triv}, \ref{ln}, \ref{lem}, 
Theorems \ref{bebis} and \ref{st} 
 (except for \ref{bebis}(4) and \ref{st}(4)) 
and in the corresponding  parts of Proposition \ref{same}.
The use of AC  seems to be essential in most examples and applications. 
 \end{remark}   

\begin{disclaimer*}
This is a preliminary report
and might contain some inaccuracies.
In particular, the author 
acknowledges that the following list of references might be
incomplete or partially inaccurate.
Henceforth the author  strongly discourages the use 
of indicators extracted from the list in decisions about individuals, attributions of funds, selections or evaluations of research projects, etc.
A more detailed disclaimer can be found at the author's web page.
\end{disclaimer*}


\begin{thebibliography}{BE}

\bibitem[B]{Br} 
S. Brandhorst, \emph{Tychonoff-Like Theorems and Hypercompact Topological Spaces}, Bachelor's thesis, Leibniz Universit\"at, Hannover, 2013.

\bibitem[BE]{BE} S. Brandhorst and M. Ern\'e,
\emph{Tychonoff-like product theorems for local topological properties},
Topology Proc. \textbf{45} (2015), 121--138.

\bibitem[H1]{H} R.-E. Hoffmann, \emph{Topological spaces admitting a ``dual''}, in 
\emph{Categorical topology},  Edited by H. Herrlich and G.
Preu\ss, Proc. int. Conf., Berlin 1978, Lect. Notes Math. 719 (1979), 157--166.

\bibitem[H2]{Henc} T. Hoshina,  \emph{Locally (P)-spaces}, in   
\emph{Encyclopedia of general topology},
edited by K. P. Hart, J. Nagata and J. E. Vaughan, Elsevier
Science Publishers, B.V., Amsterdam, (2004), 65--66.

\bibitem[L]{L} P. Lipparini, \emph{Products of sequentially compact spaces and compactness with respect to a set of filters}, arXiv:1303.0815v4 (2014).
 

\bibitem[P]{P}  G. Preu\ss, \emph{Allgemeine Topologie}, 
Hochschultext. Berlin-Heidelberg-New York: Springer-Verlag (1972).



\end{thebibliography}
\end{document}